\documentclass[11pt]{amsart}

\usepackage{amsmath}
\usepackage{amssymb}
\usepackage{graphicx}
\usepackage{enumerate}
\usepackage{color}
\newtheorem{theorem}{Theorem}[section]

\newtheorem{lemma}[theorem]{Lemma}
\newtheorem{proposition}[theorem]{Proposition}
\newtheorem{conjecture}[theorem]{Conjecture}
\theoremstyle{definition}

\newtheorem{remark}[theorem]{Remark}

\newcommand{\D}{\mathrm{d}}
\newcommand{\e}{\mathrm{e}}
\newcommand{\C}{\mathbb{C}}
\newcommand{\N}{\mathbb{N}}
\newcommand{\R}{\mathbb{R}}
\newcommand{\Z}{\mathbb{Z}}
\newcommand{\BB}{\mathcal{B}}
\newcommand{\OO}{\mathcal{O}}
\newcommand{\sumoplus}{\!\raisebox{1.5ex}{{\scriptsize $\oplus$}}}

\numberwithin{equation}{section}
\title[Geometry effects in quantum dot families]{Geometry effects in quantum dot families}

\dedicatory{To my friend Fritz Gestztesy on the occasion of his 70th birthday}

\author[P.~Exner]{Pavel Exner}

\address[P.~Exner]{Doppler Institute for Mathematical Physics and Applied Mathematics\\
Czech Technical University\\
B\v rehov\'a 7\\ 11519 Prague\\ Czechia\\ and Department of
Theoretical Physics\\ NPI\\ Academy of Sciences\\ 25068 \v{R}e\v{z}
near Prague, Czechia}
\email{{\tt exner@ujf.cas.cz}}



\begin{document}

%\begin{shortdedication}
%\emph{To my friend Fritz Gestztesy on the occasion of his 70th birthday}.
%\end{shortdedication}

\begin{abstract}
We consider Schr\"odinger operators in $L^2(\R^\nu),\, \nu=2,3$, with the interaction in the form on an array of potential wells, each on them having rotational symmetry, arranged along a curve $\Gamma$. We prove that if $\Gamma$ is a bend or deformation of a line, being straight outside a compact, and the wells have the same arcwise distances, such an operator has a nonempty discrete spectrum. It is also shown that if $\Gamma$ is a circle, the principal eigenvalue is maximized by the arrangement in which the wells have the same angular distances. Some conjectures and open problems are also mentioned.
\end{abstract}

\maketitle

%%%%%%%%%%%%%%%%%%%%%%
\section{Introduction}
\label{s:intr}

Spectral theory of Schr\"odinger operators -- the area to which Fritz contributed in a number of ways -- is a topic which may never be exhausted. In this paper we focus on what one could call \emph{guided quantum dynamics}, in other words, description of particle motion restricted in one dirrection but free in the other(s). Mathematically such systems are usually described either by Dirichlet Laplacians in tube- or layer-form regions, or alternatively by Schr\"odinger operators with a singular interaction supported by a manifold or complex of a lower dimension, see \cite{EK15} for a survey.

Recently another model attracted attention where, in contrast to the above mentioned operator classes, the confinement is `soft' being realized by a regular potential well built over a fixed curve, cf.~\cite{Ex20} and the subsequent work in \cite{KKK21, EL21, Ex22, EV23} as well as the results concerning the analogous problem about confinement in the vicinity of surfaces of a positive Gauss curvature \cite{EKP20, KK22}. One has to add that Schr\"odinger operators of this type were studied before -- see, e.g. \cite{To88, WT14} -- for a different purpose; the focus in those works was on the limit in which the `size' of the transverse confinement shrinks to zero.

The common feature of all the mentioned work is that the interaction is invariant with respect to shifts along the defining manifold; the potential depends on the distance from it only. This is, for instance, a natural model of semiconductor quantum wires. The present solid-state physics, however, makes it possible to fabricate many other objects, among which a prominent place belongs to \emph{quantum dots}, also called semiconductor nanocrystals \cite{Qwiki}. They often appear in arrays in which case a natural question concerns the electron transport, or the absence of it, in such systems. The simplest model one can use here features an array of potential wells, which is what we our going to investigate in this paper. Apart from the indicated physical motivation, an extension of the studies mentioned above  to the situation where a soft quantum waveguide has a nontrivial longitudinal structure represents a mathematical problem of an independent interest.

We analyze Schr\"odinger operators in $L^2(\R^\nu),\, \nu=2,3$, with the inter\-action term in the form of arrays of potentials wells, for simplicity assuming that each of those has a rotational symmetry. We derive two main results. The first concerns infinite arrays obtained by local perturbations of a straight family of equidistantly spaced wells, bends or deformations; using Birman-Schwinger analysis we show that they have a nonempty discrete spectrum (Theorem~\ref{thm:bent_bs}). Secondly, in analogy with \cite[Prop.~3.2.1. and Thm.~10.6]{EK15} and \cite{EL21} we consider the situation where the wells are arranged on a circle; using the Birman-Schwinger principle again we prove that the principal eigen\-value of such a Schr\"odinger operator is sharply maximized in the arrangement where the wells have the same angular distances (Theorem~\ref{thm:circleopt}). Before stating and proving these claims, we describe in the next section the setting of our task in proper terms. We will also outline relations between the present problem and spectral properties of Schr\"odinger operators with point interactions \cite{AGHH}, and we conclude the paper by listing two conjectures about the ground-state optimisation together with some other open problems about operators of this type.

%%%%%%%%%%%%%%%%%%%%%%%
\section{Preliminaries}
\label{s:prelim}

The setting we consider is simple. Given a $\rho>0$ and a real-valued function $V\in L^2(0,\rho)$ we define radial potential supported in an open ball $B_\rho(y)$ centered at a point $y\in\R^\nu,\, \nu=2,3$, as the map $x\mapsto V(\mathrm{dist}(x,y))$; with an abuse of notation we use for the latter the symbol $V$ again. Furthermore, we consider a family of points, $Y=\{y_i\}\subset\R^\nu$, finite or infinite, and such that the balls centered at them do not overlap, $\mathrm{dist}(y_i,y_j)\ge 2\rho$ if $i\ne j$, and denote by $V_i$ the potential determined by the function $x\mapsto V(x-y_i)$ in the ball $B_\rho(y_i)$. The object of our interest is then the Schr\"odinger operator
 % ------------- %
 \begin{equation} \label{hamilt}
H_{\lambda V,Y} = -\Delta - \lambda\sum_i V_i(x)
 \end{equation}
 % ------------- %
which is by our assumption about the function $V$ self-adjoint on $H^2(\R^\nu)$; we will use the shorthand $-\lambda V_Y$ for the potential term on the right-hand side of \eqref{hamilt}. Without repeating it further we will always restrict our attention to nontrivial situations when $V$ \emph{is nonzero}, and unless stated otherwise, we put $\lambda=1$, and as indicated above we also suppose that potential supports \emph{do not overlap}, $B_\rho(y_i) \cap B_\rho(y_i) =\emptyset$ for $i\ne j$. To visualise better the geometry of the set $Y$ we suppose that its points are distributed in specific ways over a curve $\Gamma\subset\R^\nu$, or alternatively over a surface $\Sigma\subset\R^3$. We are interested in relations between the form of $Y$ and the spectrum of $H_{V,Y}$, in particular about implications of variations of the geometry of the curve $\Gamma$.

If $Y$ consists of a single point, the position of the interaction plays no role and we use the abbreviated symbol $H_{\lambda V}$ for the operator \eqref{hamilt}. It is straightforward to check that $\sigma_\mathrm{ess}(H_{\lambda V})=[0,\infty)$ and the discrete spectrum, written as an ascending sequence $\{\epsilon_n\}$ with the multiplicity taken into account, is at most finite. In two dimensions it is nonempty provided $\int_0^\rho V(r)\,r\D r\ge 0$; for all small enough positive $\lambda$ there is a unique negative eigenvalue if and only if the integral is non-negative \cite{Si76}. In three dimension, the existence of bound states requires a critical interaction strength.

%%%%%%%%%%%%%%%%%%%%%%%%%%%%%%%%%%%%%%%%%%%%%%%%%%%%%%%%%
\section{Bound states in bent or locally deformed chains}
\label{s:bentchain}

In this section the set $Y$ is infinite and its points lie at a curve regarded as a continuous, piecewise $C^1$ map $\Gamma:\,\R\to\R^\nu$; without loss of generality we may suppose that the curve is unit-speed, $|\dot\Gamma|=1$, in other words, that it is parametrized by its arc length. The points of the array, which now may be denoted as $Y_\Gamma$, will be then supposed to be distributed equidistantly with respect to this variable with a spacing $a\ge 2\rho$. Note that the necessary, but in general not sufficient condition for the potential components not to overlap is $|\Gamma(s+a)-\Gamma(s)|\ge 2\rho$ for any $s\in\R$; recall that the radius of $\mathrm{supp}\,V$ is smaller than $a$ by assumption.

%%%%%%%%%%%%%%%%%%%%%%%%%%%%%%%%%%%
\subsection{The essential spectrum}
\label{ss:essential}

Consider first the geometrically trivial case case where the set $Y=Y_0$ is invariant with respect to discrete translations, i.e. the generating curve is a straight line:

 % ------------- %
 \begin{proposition} \label{prop:straightline}
Let the potentials be placed along a straight line, $\Gamma=\Gamma_0$, then $\sigma(H_{V,Y_0})\supset[0,\infty)$. If $\int_0^\rho V(r)\,r^{\nu-1}\D r\ge 0$, we have $\inf\sigma(H_{V,Y_0})<0$, and the spectrum may or may not have gaps. Their number is finite and does not exceed $\#\sigma_\mathrm{disc}(H_{V})$. This bound is saturated for the spacing $a$ large enough if $\nu=2$, in the case $\nu=3$ there may be one gap less which happens if the potential is weak, i.e. for $H_{\lambda V,Y_0}$ with $\lambda$ sufficiently small.
 \end{proposition}
 % ------------- %
\begin{proof}
Without loss of generality we may identify $\Gamma_0$ with the $x$ axis and $Y_0$ with the set $\{(ia,0\in\R^{\nu-1}): i\in\Z\}$. The inclusion $\sigma(H_{V,Y_0})\supset[0,\infty)$ is easy to check: one has to choose a disjoint family of increasing regions on which $H_V$ acts as Laplacian and construct a suitable Weyl sequence the elements of which are products of plane waves with appropriate mollifiers. To establish the band-gap structure of the negative part of the spectrum, we use Floquet decomposition, $H_{V,Y_0} = \int_\BB^\oplus H_V(\theta)\,\D\theta$ with $\BB=\big[\!-\!\frac{\pi}{a}, \frac{\pi}{a}\big)$, where the fiber $H_V(\theta)$ is an operator in $L^2(S_a)$, where $S_a:=J_a\times \R^{\nu-1}$ and $J_a:=\big(-\frac{a}{2}, \frac{a}{2}\big)$, acting as $H_V = -\Delta - V$ on the domain
 % ------------- %
 \begin{align} \label{fiberdom}
D(H_V(\theta)) & = \Big\{ \psi\in H^2(S_a):\: \psi\big(\textstyle{\frac{a}{2}},x_\perp\big) = \e^{i\theta} \psi\big(-\textstyle{\frac{a}{2}},x_\perp\big) \;\; \text{and} \\ & \partial_{x_1}\psi\big(\textstyle{\frac{a}{2}},x_\perp\big) = \e^{i\theta} \partial_{x_1}\psi\big(-\textstyle{\frac{a}{2}},x_\perp\big) \;\;\text{for all}\;\; x_\perp\in\R^{\nu-1}\Big\}, \nonumber
 \end{align}
 % ------------- %
where we use the notation $x=(x_1,x_\perp)$ for elements of $S_a$ and $\R^\nu$. The negative spectrum of $H_{V,Y_0}$ is nonempty if this is the case for some $H_V(\theta)$, and it is obvious that such spectral points can be only eigenvalues of $H_V(\theta)$. Each $H_V(\theta)$ is self-adjoint and associated with the quadratic form
 % ------------- %
 \begin{equation} \label{fiberform}
Q_{V,\theta}[\psi] := \int_{S_a} \big( |\nabla\psi(x)|^2 - V(x)|\psi(x)|^2\big)\, \D x
 \end{equation}
 % ------------- %
with the domain consisting of all $\psi\in H^1(S_a)$ that satisfy the first quasi-periodic condition in \eqref{fiberdom}. Using further the unitary transformation $\phi(x) = \e^{i\theta x_1/a}\psi(x)$, we can rewrite the form \eqref{fiberform} as the map
 % ------------- %
 $$ %\begin{equation} \label{fiberform2}
\phi \mapsto \int_{S_a} \Big( \big|\big(-i\partial_{x_1} -\textstyle{\frac{\theta}{a}}\big)\phi(x)\big|^2 + |\nabla_{x_\perp}\phi(x)|^2 - V(x)|\phi(x)|^2\Big)\, \D x
 $$ %\end{equation}
 % ------------- %
defined on $H^1(S_a)$ with periodic boundary conditions, $\phi\big(\frac{a}{2}\big) = \phi\big(\!-\!\frac{a}{2}\big)$. From here one can check that the eigenvalues of $H_V(\theta)$, if they exist, are continuous functions of $\theta$, and their ranges constitute the spectral bands. Moreover, the lower and upper band edges correspond respectively to the symmetric and antisymmetric solutions, $\psi(x) = \pm\psi(-x)$, due to the smoothness referring to the Neumann and Dirichlet condition at the slab boundary, while the bracketing argument \cite[Sec.~XIII.15]{RS} applied to \eqref{fiberform} gives the bounds
 % ------------- %
 \begin{equation} \label{bracket}
H_{V,a}^\mathrm{N} \le H_V(\theta) \le H_{V,a}^\mathrm{D},\quad \theta\in\BB,
 \end{equation}
 % ------------- %
where $H_{V,a}^\mathrm{N/D}$ are the operators acting as $-\Delta + V$ on functions of $H^2(S_a)$ satisfying the Neuman and Dirichlet conditions, respectively, at the boundary of the slab $S_a$. By minimax principle, this means that the $j$th spectral band is squeezed between the $j$th eigenvalues of the two operator provided those exist; if such an eigenvalue exists for $H_{V,a}^\mathrm{N}$ but not for $H_{V,a}^\mathrm{D}$ the upper bound is replaced by zero. Another simple application of the bracketing argument shows that the estimating eigenvalues are monotonous with respect to $a$, the lower (Neumann) being increasing with respect to $a$, the upper decreasing, so that the bands shrink as $a$ increases. Furthermore, we note that the discrete spectrum of the two operators is the same as that of $\tilde{H}_{V,a}^\mathrm{N/D}$ obtained from $H_V$ by adding the Neumann/Dirichlet condition at $x_1=\pm\frac{a}{2}$ since the `outer' part of these operators are positive. Increasing the spacing of the added conditions we arrive eventually to the same eigenvalue equation, hence the bands shrink to the eigenvalues of $H_V$ as $a\to\infty$; this yields the last claim.

To prove the sufficient condition for the existence of negative spectrum it is enough to find a trial function which makes the form $Q_{V,0}$ of \eqref{fiberform} negative. We can use, for instance, the functions
 % ------------- %
 \begin{equation} \label{trial2}
\chi_{b,c}(x) = \left\{ \begin{array}{ccl} 1 & \quad\dots\quad & |x_\perp|\le b \\ \frac{c-x_\perp}{\!c-b} & \quad\dots\quad & b \le |x_\perp| \le c \\ 0 & \quad\dots\quad & |x_\perp| \ge c \end{array} \right.
 \end{equation}
 % ------------- %
independent of $x_1$ if $\nu=2$, and
 % ------------- %
 \begin{equation} \label{trial3}
\chi_{b,c}(x) = \left\{ \begin{array}{ccl} 1 & \quad\dots\quad & |x_\perp|\le b \\ -\ln\frac{|x_\perp|}{c} \big( \ln\frac{c}{b} \big)^{-1} & \quad\dots\quad & b \le |x_\perp| \le c \\ 0 & \quad\dots\quad & |x_\perp| \ge c \end{array} \right.
 \end{equation}
 % ------------- %
if $\nu=3$. Choosing $b>\rho$ we ensure that the supports of $V$ and $\nabla\chi_{b,c}$ are disjoint. The potential term in $Q_{V,0}[\chi_{b,c}]$ equals $\inf\sigma(H_{V,Y_0)})$ being negative by assumption, and it is easy to chect that the kinetic term can be made arbitrarily small by putting $c$ sufficiently large; this concludes the proof.
\end{proof}

 % ------------- %
 \begin{remark} \label{rem:critical}
The fact that $\inf\sigma(H_{V,Y_0)})<0$ holds whenever the potential $V$ is attractive in the mean is not in contradiction with the need of critical strength to achieve $\inf\sigma(H_V)<0$ in the three-dimensional case; note that the lower edge of the spectrum indicated in the proof converges then to zero as $a\to\infty$. We also note that the spectrum is absolutely continuous, however, we will not need this property in the following.
 \end{remark}
 % ------------- %

Our aim is to find what happens with the spectrum, if $\Gamma$ is bent or locally deformed; to make things simpler we assume that the curved part is finite and the halfline asymptotes of $\Gamma$ are either not parallel, or if they are, they point in the opposite directions. Then the continuous spectrum does not change, however, we limit ourselves to the claim we will need in the following.
 % ------------- %
 \begin{proposition} \label{prop:bentline}
Suppose that the curve $\Gamma$ is straight outside a compact set and $|\Gamma(s)-\Gamma(-s)|\to\infty$ holds as $|s|\to\infty$, then $\sigma_\mathrm{ess}(H_{V,Y}) \supset \sigma(H_{V,Y_0})$ and the thresholds of the two sets coincide.
 \end{proposition}
 % ------------- %
\begin{proof}
The inclusion $\sigma(H_{V,Y})\supset[0,\infty)$ is checked as in the previous case. To prove that also the negative part of the straight array essential spectrum is preserved, we use again the Weyl criterion by which $\mu$ belongs to $\sigma(H_{V,Y})$ if and only if there is a sequence $\{\psi_n\}\subset D(H_{V,Y})$ such that
 % ------------- %
 \begin{equation} \label{Weyl3}
\lim_{n\to\infty}\| (H_{V,Y}-\mu)\psi_n\| = 0.
 \end{equation}
 % ------------- %
For $H_{V,Y_0}$ such a sequence can be constructed explicitly, its elements being products of the generalized eigenfunction of $H_{V,Y_0}$ corresponding to the eigenvalue $\mu$ of $H_V(\theta)$ for an appropriate $\theta\in\BB$ with suitable mollifiers in the $x_1$ variable; without loss of generality the latter can be chosen to have disjoint supports. This can be used to construct a Weyl sequence for $H_{V,Y}$ in the form $\{\psi_n\chi_n\}$ where $\chi_n$ are transverse mollifiers of the type \eqref{trial2} or \eqref{trial3} for $\nu=2,3$, respectively. By assumption the radius of $\mathrm{supp}\,\chi_n$ can be made arbitrarily large, and consequently, this mollifier influence on $\| (H_{V,Y}-\mu)\psi_n\|$ arbitrarily small if the longitudinal mollifier is supported far enough from the curved part of $\Gamma$; this yields the inclusion $\sigma(H_{V,Y_0})\subset \sigma(H_{V,Y})$, and as the supports of $\psi_n$'s are disjoint, we also have $\sigma(H_{V,Y_0})\subset \sigma_\mathrm{ess}(H_{V,Y})$.

In order to prove the remaining claim, we divide the plane into four closed regions, $\Sigma_+ \cup \Sigma_- \cup \Sigma_0 \cup \Sigma_\mathrm{ext}$. The first two are semiinfinte rectangular strips of width $L$, the axes of which coincide with the straight parts of $\Gamma$; the perpendicular boundary of each of them passes through some of the points $y_i$. The third one is a polygon, to which $\Sigma_\pm$ are attached, and $\Sigma_\mathrm{ext}$ is the rest of the plane. Imposing Neumann conditions on the common boundaries of the regions, we get by \cite[Sec.~XIII.15]{RS} an operator estimating $H_{V,Y}$ from below. Its part in $\Sigma_0$ is finite, thus irrelevant from the viewpoint of the essential spectrum. Likewise, the part in $\Sigma_\mathrm{ext}$ is a positive operator which tells us nothing about $\inf\sigma_\mathrm{ess}(H_{V,Y})$. Since the `lids' of $\Sigma_\pm$ are Neumann, the spectral threshold of these operator parts coincide with that of the straight array confined symmetrically to a Neumann strip of width $L$ and that one, \emph{mutatis mutandis}, in analogy with \eqref{bracket} is the ground state eigenvalue of $H_{V,a}^{\mathrm{N},L}$, the restriction of our operator to the rectangle $S_{a,L}=\big(-\frac{a}{2}, \frac{a}{2}\big)\times \big(-\frac{L}{2}, \frac{L}{2}\big)$ with Neumann boundary.

What is important, the plane decomposition can be chosen to have $L$ arbitrarily large. It may require to make $\Sigma_0$ appropriately large, but it does not matter as long as it is finite; the task then reduces to checking that $\lim_{L\to\infty} \inf\sigma(H_{V,a}^{\mathrm{N},L}) = \inf\sigma(H_{V,a}^\mathrm{N})$. Such limits for Neumann Schr\"odinger operators have been studied only in one dimension or for domains of particular shapes \cite{DH93}, however, the rectangular shape of the support makes it possible to check the indicated limit directly. Passing from the coordinate $y$ in $S_{a,L}$ to $u:= \frac{2L}{\pi}\,\tan\frac{\pi y}{2L}$, we see that the problem is equivalent to finding the ground-state eigenvalue of the operator
 % ------------- %
 $$
\tilde{H}_{V,a}^{\mathrm{N},L} := -\frac{\D^2}{\D x^2} - \Big(1+\Big(\frac{\pi u}{2L}\Big)^2\Big) \frac{\D}{\D u} \Big(1+\Big(\frac{\pi u}{2L}\Big)^2\Big) \frac{\D}{\D u} + V\Big(x,\frac{2L}{\pi}\,\tan\frac{\pi y}{2L}\Big)
 $$
 % ------------- %
on $S_a$ in the limit $L\to\infty$. We can treat it as a perturbation of $H_{V,a}^\mathrm{N}$ since for $V=0$ the relation $\tilde{H}_{0,a}^{\mathrm{N},L} \to H_{0,a}^\mathrm{N}$ in the generalized strong resolvent sense \cite{We84} is established directly and the potential perturbation is relatively bounded; this yields the desired result.
\end{proof}

%%%%%%%%%%%%%%%%%%%%%%%%%%%%%%%%%%
\subsection{The discrete spectrum}
\label{ss:geom_bs}

Now we are going to suppose that the potential is \emph{purely attractive}, $V\ge 0$, and show that geometric perturbations do then give rise to a noempty spectrum below $\mu_0:= \inf\sigma(H_{V,Y})$. We will employ the Birman-Schwinger principle; for a rich bibliography concerning this remarkable tool we refer to \cite{BEG22}. To this aim we define for any $z\in\C\setminus\R_+$ the operator in $L^2(\R^\nu)$,
 % ------------- %
 \begin{equation} \label{BSop}
K_{V,Y}(z) := V_Y^{1/2} (-\Delta-z)^{-1} V_Y^{1/2}\,;
 \end{equation}
 % ------------- %
we are particularly interested in the negative spectral parameter value, $z=-\kappa^2$ with $\kappa>0$. In view of our assumptions about the potential the nontrivial part of $K_{V,Y}(-\kappa^2)$ is positive and maps $L^2(\mathrm{supp}\,V_Y) \to L^2(\mathrm{supp}\,V_Y)$. Since the supports of the potentials $V_i$ are disjoint by assumption, we have
 % ------------- %
 $$ %\begin{equation} \label{BSspace}
L^2(\mathrm{supp}\,V_Y) = \sum_i\sumoplus\: L^2(B_\rho(y_i))
 $$ %\end{equation}
 % ------------- %
and using this orthogonal sum decomposition we can write the Birman-Schwinger operator \eqref{BSop} in the `matrix' form with the `entries'
 % ------------- %
 \begin{equation} \label{BSmatrix}
K_{V,Y}^{(i,j)}(-\kappa^2) := V_i^{1/2} (-\Delta+\kappa^2)^{-1} V_j^{1/2}
 \end{equation}
 % ------------- %
mapping $L^2(B_\rho(y_j))$ to $L^2(B_\rho(y_i))$. The Birman-Schwinger principle allows us to determine eigenvalues of $H_{V,Y}$ by inspection of those of $K_{V,Y}(-\kappa^2)$:
 % ------------- %
\begin{proposition} \label{prop:BS}
$z\in\sigma_\mathrm{disc}(H_{V,Y})$ holds if and only if $\,1\in\sigma_\mathrm{disc}(K_{V,Y}(z))$ and the dimensions of the corresponding eigenspaces coincide. The
operator $K_{V,Y}(-\kappa^2)$ is bounded for any $\kappa>0$ and the function $\kappa \mapsto K_{V,Y}(-\kappa^2)$ is continuously decreasing in $(0,\infty)$ with $\lim_{\kappa\to\infty}\|K_{V,Y}(-\kappa^2)\|=0$.
\end{proposition}
 % ------------- %
\begin{proof}
The first claim is a particular case of a more general and commonly known result, see, e.g., \cite{BGRS97}. Using the explicit form of $(-\Delta-z)^{-1}$ as the integral operator with the kernel $(x,x')\mapsto \frac{1}{2\pi} K_0(\kappa|x-x'|)$ and $\frac{\e^{-\kappa|x-x'|}}{4\pi|x-x'|}$ for $\nu=2,3$, respectively, we can check that $K_{V,Y}(-\kappa^2)$ is bounded if $V\in L^2$. Using Sobolev inequality \cite[Sec.~IX.4]{RS} we infer that each $K_{V,Y}^{(i,j)}(-\kappa^2)$ has a finite Hilbert-Schmidt norm, uniformly in $i,j$, if $\nu=3$. To make the same conclusion for $\nu=2$ one has to use in addition the fact that $|K_0(\kappa r)| \le cr^{-1}$ holds on $[0,2\rho]$ for a fixed $\kappa>0$ and some $c>0$. The operator $K_{V,Y}(-\kappa^2) = \sum_{i,j} K_{V,Y}^{(i,j)}(-\kappa^2)$ is no longer compact, of course, but due to the uniformity the boundedness persists.

The continuity in $\kappa$ follows from the functional calculus and  we have
 % ------------- %
 $$ %\begin{equation} \label{monotonicity}
\frac{\D}{\D\kappa} (\psi,V_Y^{1/2} (-\Delta+\kappa^2)^{-1}\, V_Y^{1/2}\psi) = -2\kappa (\psi,V_Y^{1/2} (-\Delta+\kappa^2)^{-2}\, V_Y^{1/2}\psi) < 0
 $$ %\end{equation}
 % ------------- %
for any nonzero $\psi\in L^2(\mathrm{supp}\,V_Y)$ which implies, in particular, the norm monotonicity. It follows from the dominated convergence theorem that $\lim_{\kappa\to\infty} \|K_{V,Y}^{(i,i)}(-\kappa^2)\|_2=0$ holds for the `diagonal' operators, uniformly in $i$. Using further the fact that $\mathrm{dist}(y_i,y_j) \ge \delta:= a-2\rho>0,\: i\ne j$, we get
 % ------------- %
 $$ %\begin{equation} \label{monotonicity}
\|K_{V,Y}^{(i,j)}(-\kappa^2)\| \le \|K_{V,Y}^{(i,j)}(-\kappa^2)\|_2 \le \frac{\e^{-\kappa\delta}}{4\pi\delta}\, \|K_{V,Y}^{(i,i)}(-\kappa^2)\|_2
 $$ %\end{equation}
 % ------------- %
for the `non-diagonal' opertors if $\nu=3$, and a similar estimate with the right-hand side factor replaced by $\frac{1}{2\pi}K_0(\kappa\delta)$ if $\nu=2$.
\end{proof}
 % ------------- %

 % ------------- %
 \begin{remark} \label{rem:BSessential}
Applying the Birman-Schwinger principle to the fiber operators in the decomposition $H_{V,Y_0} = \int_\BB^\oplus H_V(\theta)\,\D\theta$ one can check that the spectrum of $K_{V,Y_0}^{(i,j)}(-\kappa^2)$ has the band-gap structure and the function $\kappa\mapsto \sup\sigma(K_{V,Y_0}(-\kappa^2))$ is decreasing being equal to one at $\kappa_0=\sqrt{-\epsilon_0}$. By Proposition~\ref{prop:bentline} the essential spectrum threshold is preserved by the considered geometric perturbations, hence the  function $\kappa\mapsto \sup\sigma_\mathrm{ess}(K_{V,Y}(-\kappa^2))$ has the same properties; note that one can apply the BS principle to the essential spectrum directly using the spectral shift function \cite{Pu11}.
 \end{remark}
 % ------------- %

Now we are in position to state the main result of this section:
 % ------------- %
\begin{theorem} \label{thm:bent_bs}
Assume that $\Gamma$ satisfying the assumptions of Proposition~\ref{prop:bentline} is not a straight line and $V\ge 0$, then $\inf\sigma(H_{V,Y})<\epsilon_0$, and consequently, we have $\sigma_\mathrm{disc}(H_{V,Y})\ne\emptyset$.
\end{theorem}
 % ------------- %
\begin{proof}
In view of Proposition~\ref{prop:BS} we have to show that there is a $\kappa>\sqrt{-\epsilon_0}$ such that $K_{V,Y}(-\kappa^2)$ has eigenvalue one. By Remark~\ref{rem:BSessential} such a spectral point can be an eigenvalue of finite multiplicity only, and Proposition~\ref{prop:BS} tells us that any such eigenvalue is a decreasing function of $\kappa$ which tends to zero as $\kappa\to\infty$. To prove the theorem, it is thus sufficient to check that $\sup\sigma(K_{V,Y}(-\kappa_0^2)) > 1 =\sup\sigma_\mathrm{ess}(K_{V,Y}(-\kappa_0^2))$. Note that by Proposition~\ref{prop:BS} and minimax principle the function $\kappa\mapsto \sup\sigma_\mathrm{ess}(K_{V,Y}(-\kappa^2)$ is continuously decreasing as well. One cannot exclude situations when an eigenvalue branch will be absorbed in the continuum as $\kappa$ increases, however, in view of the indicated monotonicity it must first cross level one.

We are going to use the fact that our geometric perturbations are sign-definite -- in the mean sense -- and construct a trial function $\psi$ such that
 % ------------- %
 \begin{equation} \label{BStrial}
(\phi, K_{V,Y}(-\kappa_0^2)\phi) - \|\phi\|^2 >0.
 \end{equation}
 % ------------- %
The first expression on the right-hand side can be rewritten as
 % ------------- %
 $$ %\begin{equation} \label{BSform}
\int_{\R^\nu\times\R^\nu}  \overline\phi(x) V_Y^{1/2}(x) (-\Delta+\kappa_0^2)^{-1}(x,x') V_Y^{1/2}(x')\phi(x')\,\D x\,\D x',
 $$ %\end{equation}
 % ------------- %
or more explicitly using the operators \eqref{BSmatrix} as
 % ------------- %
 $$ %\begin{equation} \label{BSform}
\sum_{i,j\in\Z} \int_{B_\rho(y_i)\times B_\rho(y_j)}  \overline\phi(x) V_i^{1/2}(x) (-\Delta+\kappa_0^2)^{-1}(x,x') V_j^{1/2}(x')\phi(x')\,\D x\,\D x'.
 $$ %\end{equation}
 % ------------- %
Denote now by $\phi_0$ the generalized eigenfunction of $K_{V,Y}(-\kappa_0^2)$ corresponding to the spectral threshold of the straight chain $Y_0$; as this function  is the product of the corresponding generalized eigenfunction of $H_{V,Y_0}$ and $V_Y^{1/2}$, it is periodic and without loss of generality we may suppose that it is real-valued and positive. What matters are the restrictions of $\phi_0$ to the balls supporting the potential, $\phi_{0,i}= \phi_0 \upharpoonright B_\rho(y_i)$, which are shifted copies of the same function, $\phi_{0,i}(\xi)=\phi_0(\xi+y_i)$ for $\xi\in B_\rho(0)$. Recall that we identified $Y_0$ with the set $\{(ia,0\in\R^{\nu-1}): i\in\Z\}$, then functions $\phi_{0,i}$ are even with respect to the ball centers in the direction of the chain axis, and have rotational symmetry with respect to it (for $\nu=2$ this means being even also transversally), in other words, $\phi_{0,i}(-\xi)=\phi_{0,i}(\xi)$ holds for $\xi\in B_\rho(0)$.

As it is common in such situations \cite{Ex20}, we use the function $\phi_0$ as the starting point for construction of the sought trial function. Using it we construct for a given $Y$ the functions $\phi^Y_0$ as an `array of beads': its values in $B_\rho(y_i)$ would coincide with $\phi_{0,i}$ rotated in such a way that the axis of $\phi_{0,i}$ agrees with the tangent to $\Gamma$ at the point $y_i$; for $Y=Y_0$ we drop the superscript $Y$.
To make such a function an $L^2$ element, we need a suitable family of mollifiers; we choose it in the form
 % ------------- %
 \begin{equation} \label{BSmollif}
h_n(x) = \frac{1}{2n+1}\,\chi_{M_n}(x),\quad n\in\N.
 \end{equation}
 % ------------- %
where $M_n:=\{x:\: \mathrm{dist}(x,\Gamma\upharpoonright [-(2n+1)a/2,(2n+1)a/2])\le\rho\}$ is a $2\rho$-wide closed tubular neighborhood of the $(2n+1)a$-long arc of $\Gamma$. We have to ensure that the positive contribution from such a cut-off can be made arbitrarily small. This is indeed the case:
 % ------------- %
\begin{lemma} \label{l:BSmollif}
$(h_n\phi_0^Y, K_{V,Y}(-\kappa_0^2)h_n\phi_0^Y) - \|h_n\phi_0^Y\|^2=\OO(n^{-1})$ as $n\to\infty$.
\end{lemma}
 % ------------- %
\begin{proof}
Since $\phi_0$ is periodic along the chain, one obtains for the second term the following expression,
 % ------------- %
 $$ %\begin{equation} \label{BSsecond}
\|h_n\phi_0\|^2 = \frac{1}{(2n+1)^2} \int_{B_\rho(0)} |\phi_0(x)|^2\,\D x.
 $$ %\end{equation}
 % ------------- %
Using the fact that the function \eqref{BSmollif} is constant on its support, we get
 % ------------- %
 \begin{align*}
& (h_n\phi_0^Y, K_{V,Y}(-\kappa_0^2)h_n\phi_0^Y) = \frac{1}{(2n+1)^2} \sum_{|i|\le n} \sum_{|j|\le n} \int_{B_\rho(y_i)} \D x\, \phi_{0,i}^Y(x) \\[.15em]
& \qquad\qquad \times \int_{B_\rho(y_j)} V_Y^{1/2}(x)\, (-\Delta+\kappa_0^2)^{-1}(x,x')\, V_Y^{1/2}(x') \phi_{0,i}^Y(x')\,\D x' \\[.3em]
& \quad = \frac{1}{(2n+1)^2} \sum_{|i|\le n} \sum_{|j|\le n} \int_{B_\rho(0)} \D\xi\, \phi_{0}(\xi) \\[.15em]
& \qquad\qquad \times \int_{B_\rho(0)} V^{1/2}(\xi)\, (-\Delta+\kappa_0^2)^{-1}(\xi,\xi'+y_j-y_i)\, V^{1/2}(\xi') \phi_{0}(\xi')\,\D\xi'
 \end{align*}
 % ------------- %
By assumption, $\phi_0$ -- now understood as a periodic function on $\R$ -- is the generalized eigenfunction of $K_{V,Y_0}(-\kappa_0^2)$ corresponding to the spectral threshold; this makes it possible to rewrite the right-hand side of the last relation as
 % ------------- %
 \begin{align*}
& \frac{1}{(2n+1)^2} \int_{B_\rho(0)} |\phi_0(x)|^2\,\D x - \frac{1}{(2n+1)^2} \sum_{|i|\le n} \int_{B_\rho(0)} \D\xi\, \phi_{0}(\xi) \\[.3em]
& \times \sum_{|j|>n} \int_{B_\rho(0)} V^{1/2}(\xi)\, (-\Delta+\kappa_0^2)^{-1}(\xi,\xi'+y_j-y_i)\, V^{1/2}(\xi') \phi_0(x')\,\D\xi'
 \end{align*}
 % ------------- %
For a straight chain we have $|y_j-y_i|=a|j-i|$, for a curved one the distance under the assumptions of Proposition~\ref{prop:bentline} also increases linearly as $|j-i|\to\infty$. Given the fact that the resolvent kernel is asymptotically exponentially decreasing, we see that the second sum converges for a fixed $Y$ and has a bound independent of $i$ which yields the sought result.
\end{proof}

In view of the lemma and relation \eqref{BStrial} one has therefore to check that
 % ------------- %
 $$ %\begin{equation} \label{}
(h_n\phi_0, K_{V,Y}(-\kappa_0^2)h_n\phi_0) - (h_n\phi_0, K_{V,Y_0}(-\kappa_0^2)h_n\phi_0) > 0.
 $$ %\end{equation}
 % ------------- %
holds for all $n$ large enough, or equivalently, that
 % ------------- %
 $$ %\begin{equation} \label{}
\lim_{n\to\infty} (h_n\phi_0^Y, K_{V,Y}(-\kappa_0^2)h_n\phi_0^Y) - (h_n\phi_0 K_{V,Y_0}(-\kappa_0^2)]h_n\phi_0) > 0.
 $$ %\end{equation}
 % ------------- %
In view of \eqref{BSmatrix}, in turn, this will be true if we prove that
 % ------------- %
 \begin{equation} \label{BStrial2}
(\phi_0, [K^{(i,j)}_{V,Y}(-\kappa^2)- K^{(i,j)}_{V,Y_0}(-\kappa^2)]\phi_0) \ge 0
 \end{equation}
 % ------------- %
holds any $\kappa>0$ and all $i,j\in\Z$ being \emph{positive} for some of them. In the last relation we allow for an abuse of notation writing the first part as the matrix element between the functions $\phi_0$ keeping in mind, of course, that the axes of its components in $B_\rho(y_i)$ and $B_\rho(y_j)$ are in general not parallel. Naturally, the left-hand side of \eqref{BStrial2} is zero for $i=j$ or for $i,j$ such that the segment of $\Gamma$ between $y_i$ and $y_j$ is straight. If $Y\ne Y_0$, however, there is a pair of indices for which this is not the case, $|y_i-y_j|<|i-j|a$, in fact, infinitely many such pairs. Was the potential a point interaction as in \cite{Ex01}, the result would follow immediately from the monotonicity of the resolvent kernel, but the problem is more subtle here because bending of the chain, even a weak one, may cause some distances between points of potential supports outside the ball centers to \emph{increase}.

Denoting the resolvent kernel by $G_{i\kappa}$ for the sake of brevity, we can write the left-hand side of \eqref{BStrial2} explicitly as
 % ------------- %
 \begin{align*}
& \int_{B_\rho(0)} \int_{B_\rho(0)} \phi_0(\xi) V^{1/2}(\xi) \big[ G_{i\kappa}(y_i-y_j +\xi-\xi') - G_{i\kappa}(y_i^{(0)}\!-y_j^{(0)}\! +\xi-\xi') \big] \\[.3em] & \qquad \times V^{1/2}(\xi') \phi_0(\xi') \D\xi\,\D\xi' \\[.3em]
& = \frac12 \int_{B_\rho(0)} \int_{B_\rho(0)} \phi_0(\xi) V^{1/2}(\xi) \big[ G_{i\kappa}(y_i-y_j +\xi-\xi') - G_{i\kappa}(y_i^{(0)}\!-y_j^{(0)}\! +\xi-\xi') \\[.3em] & \qquad + G_{i\kappa}(y_i-y_j -\xi+\xi') - G_{i\kappa}(y_i^{(0)}\! -y_j^{(0)}\! -\xi+\xi')\big] V^{1/2}(\xi') \phi_0(\xi') \D\xi\,\D\xi,
 \end{align*}
 % ------------- %
where we have used the fact that $\phi_0(\xi) V^{1/2}(\xi) = \phi_0(-\xi) V^{1/2}(-\xi)$. The integration over $\xi$ can be split into the transversal and longitudinal part with respect to the vector $y_i-y_j$, namely $\int_{B_\rho(0)} \D\xi = \int_{-\rho}^\rho \D\xi_\perp \int_{\scriptsize -\sqrt{\rho^2-s_\perp^2}}^{\scriptsize \sqrt{\rho^2-s_\perp^2}} \D\xi_{||}$, and similarly for $\xi'$. What is important is the behavior of the square bracket at the longitudinal integration. We observe that not only the function $G_{i\kappa}(\cdot)$ is convex, but the same is true for $G_{i\kappa}(|y_i-y_j|+\cdot) - G_{i\kappa}(|y_i^{(0)}\!-y_j^{(0)}|+\cdot)$ as long as $|y_i-y_j| < |y_i^{(0)}\!-y_j^{(0)}|$, and in that case the square bracket can be estimated by virtue of Jensen's inequality from below by
 % ------------- %
 $$ %\begin{equation} \label{}
G_{i\kappa}(|y_i-y_j|) - G_{i\kappa}(|y_i^{(0)}\!-y_j^{(0)}|) > 0.
 $$ %\end{equation}
 % ------------- %
In combination with the positivity of $\phi_0V^{1/2}$ this proves that the right-hand side of \eqref{BStrial2} is positive whenever $|y_i-y_j|<|i-j|a$; this concludes the proof of the theorem.
\end{proof}
 % ------------- %
 \begin{remark} \label{rem:asymmetry}
The symmetry of the potential $V$ played an important role in the proof. In its absence the above argument would work only if the deformation of $\Gamma$ is strong enough to diminish \emph{all} the distances between the points of the considered pairs of balls, for instance, if $|y_i-y_{i+1}| <a-2\rho$ holds for neighboring balls. Such a condition is clearly not optimal and the harder to fulfill the larger the ratio $\frac{\rho}{a}$ is; we postpone the discussion of this question to a later publication.
 \end{remark}
 % ------------- %

%%%%%%%%%%%%%%%%%%%%%%%%%%%%%%%%%
\section{Shrinking the potential}
\label{s:shrinking}

If the potential $V$ is strongly localized one may think about replacing $H_{V,Y}$ by a singular Schr\"odinger operator. Properties of point-interaction Hamiltonians are well known and nicely summarized in the classical monograph \cite{AGHH}. These operators can be introduced by several equivalent ways; one of them is based on self-adjoint extensions, starting from restriction of the Laplacian to $C_0^\infty(\R^\nu\setminus Y)$. In dimensions $\nu=2,3$ the resulting operator has deficiency indices $(\#Y,\#Y)$ \cite{BG85}; among its numerous self-adjoint extensions one focuses on the \emph{local} ones characterized -- in the present situation when all the interactions are the same -- by the boundary conditions
 % ------------- %
\begin{equation} \label{pi-bc}
L_1(\psi,y_j) -\alpha L_0(\psi,y_j)=0, \quad \alpha\in\R,
\end{equation}
 % ------------- %
coupling the generalized boundary values
 % ------------- %
\begin{align*}
L_0(\psi,y):=& \lim_{|x-y|\to 0}\, {\psi(x)\over
\phi_\nu(x\!-\!y)}, \\ L_1(\psi,y):=& \lim_{|x-y|\to 0}
\bigl\lbrack \psi(x)- L_0(\psi,y)\, \phi_\nu(x-y) \bigr\rbrack,
\end{align*}
 % ------------- %
where $\phi_\nu$ are the appropriate fundamental solutions,
 % ------------- %
$$
\phi_2(x)= -{1\over2\pi}\, \ln|x|, \quad \phi_3(x) = {1\over 4\pi|x|},
$$
 % ------------- %
for $\nu=2,3$, respectively. These point interactions are non-additive perturbations of the free Hamiltonian; the latter obviously corresponds to $\alpha=\infty$. Following \cite{AGHH} we employ the symbol $-\Delta_{\alpha,Y}$ for the singular operators defined by boundary conditions \eqref{pi-bc}.

Approximation of point interactions in dimensions $\nu=2,3$ is not an easy matter; it is well known that such a limit is generically trivial. There are nevertheless situations when one can make sense of such a limit:
 % ------------- %
\begin{proposition} \label{thm:piapprox}
Let $\nu=3$ and assume that $H_V$ has a zero-energy resonance with which one can associate a solution $f\in L^2_\mathrm{loc}(\R^\nu)$ of the equation $V^{1/2}(-\Delta)^{-1}V^{1/2}f=f$, then the family of operators
 % ------------- %
$$ %\begin{equation} \label{approxop}
H_{V_\varepsilon,Y} := -\Delta + \frac{\mu(\varepsilon)}{\varepsilon^2}\, \sum_i V\big(\textstyle{\frac{\cdot\,-\,y_i}{\varepsilon}}\big),
$$ %\end{equation}
 % ------------- %
where $\mu$ is real analytic in the vicinity of zero and such that $\mu(0)=1$, converges in the norm resolvent sense to $-\Delta_{\alpha,Y}$ with $\alpha:= -\mu'(0)|(V_Y^{1/2},f)|^{-2}$.
\end{proposition}
 % ------------- %
\begin{proof}
Since the points of $Y$ do not accumulate, $\inf_{i\ne j} |y_i-y_j|\ge 2a>0$, the claim follows from the analysis presented in \cite{AGHH}, in particular, from Theorems II.1.2.1 and III.1.2.1 there.
\end{proof}
 % ------------- %

In the two-dimensional case zero-energy resonances of $H_V$ are again crucial. The scaled-potential approximation is worked out in \cite{AGHH} for single point interaction but it can be extended to more complicated sets $Y$ similarly as for $\nu=3$; the resulting parameter $\alpha$ again depends on how exactly the coupling constant is scaled in the vicinity of the resonance.

For point-interaction Hamiltonians $-\Delta_{\alpha,Y}$ one can also ask about the implications that a nontrivial geometry would have for the spectrum. What one finds in this case is consistent with the results of the previous section: a bend or a local deformation of a straight periodic array, which shortens the Euclidean distances, lowers the spectral threshold, and if $Y$ is asymptotically straight in a suitable sense so that the essential spectrum is preserved, isolated eigenvalues emerge again \cite{Ex01}. At the same time, the approximation of $-\Delta_{\alpha,Y}$ by Schr\"odinger operators with scaled potential does not require spherical symmetry of the potential $V$ which, similarly as Remark~\ref{rem:asymmetry}, gives a hint that assumptions of Theorem~\ref{thm:bent_bs} might be weakened.

%%%%%%%%%%%%%%%%%%%%%%%%%%%%%%%%%%%
\section{Ground state optimization}
\label{s:opt}

Let us return to arrays of regular potentials, this time finite ones, and change slightly the setting. We consider the two-dimensional situation and fix the curve $\Gamma$ which will be now a \emph{circle} of radius $R$ on which we place the centers of the disks $B_\rho(y_i)$; without loss of generality we may identify the circle center with the origin of the coordinates. The only restriction imposed is that they must not overlap, that is, $\rho\le R\sin\frac{\pi}{N}$, where $N:=\#Y$.

We are interested in the configuration which makes the principal eigenvalue of $H_{V,Y}$ maximal. It appears that this happens if $Y$ has the full symmetry with respect to the discrete rotations:
 % ------------- %
\begin{theorem} \label{thm:circleopt}
Up to rotations, $\epsilon_1(H_{V,Y})=\inf\sigma(H_{V,Y})$ is uniquely maximized by the configurations in which all the neighboring points of $Y$ have the same angular distance $\frac{2\pi}{N}$.
\end{theorem}
 % ------------- %
\begin{proof}
The potential is compactly supported, so the negative spectrum of $H_{V,Y}$ is now discrete and finite, and the ground state $\epsilon_1(H_{V,Y})$ is a simple eigenvalue. We denote by $Y_\mathrm{sym}$ the symmetric array in which all the neighboring points have the same angular distances. The real-valued eigenfunction $\psi_\mathrm{sym}$ associated with $\epsilon_1(H_{V,Y_\mathrm{sym}})$ has the appropriate symmetry: in polar coordinates we have $\psi_\mathrm{sym}(r,\varphi) = \psi_\mathrm{sym}(r,\varphi + \frac{2\pi n}{N})$ for any $n\in\Z$.

We use the Birman-Schwinger principle again and denote by $\phi_\mathrm{sym}$ the eigenfunction corresponding to the largest eigenvalue of $K_{V,Y_\mathrm{sym}}(\epsilon_\mathrm{sym})$, where $\epsilon_\mathrm{sym}= \inf\sigma(H_{V,Y_\mathrm{sym}})$. It also has the symmetry with respect to rotations on multiples of the angle $\frac{2\pi}{N}$, and as in the proof of Theorem~\ref{thm:bent_bs} we may suppose that it is real-valued and positive. Referring to the monotonicity of $K_{V,Y}(\cdot)$ stated in Proposition~\ref{prop:BS}, in order to show that $\epsilon_1(H_{V,Y})<\epsilon_1(H_{V,Y_\mathrm{sym}})$ holds whenever $Y\ne Y_\mathrm{sym}$, modulo discrete rotations, one has to check the inequality $\max\sigma(K_{V,Y}(\epsilon_\mathrm{sym})) > \max\sigma(K_{V,Y_\mathrm{sym}}(\epsilon_\mathrm{sym}))$, and to this goal it is sufficient to find a trial function $\phi$ such that
 % ------------- %
 \begin{equation} \label{BSring}
(\phi, K_{V,Y}(-\kappa_\mathrm{sym}^2)\phi) - \|\phi\|^2 >0, \quad \kappa_\mathrm{sym}=\sqrt{-\epsilon_\mathrm{sym}}.
 \end{equation}
 % ------------- %

A general configuration $Y$ of point on the circle is characterized by the family of angles $\theta_i,\, i=1,\dots,N$ satisfying $\sum_{i=1}^N \theta_j=2\pi$ as sketched in Fig.~\ref{fig:sectors} for $N=5$.
 % -------------- %
\begin{figure}[h!]
\centering
    \includegraphics[clip, trim=5cm 15cm 5cm 4cm, angle=0, width=0.67\textwidth]{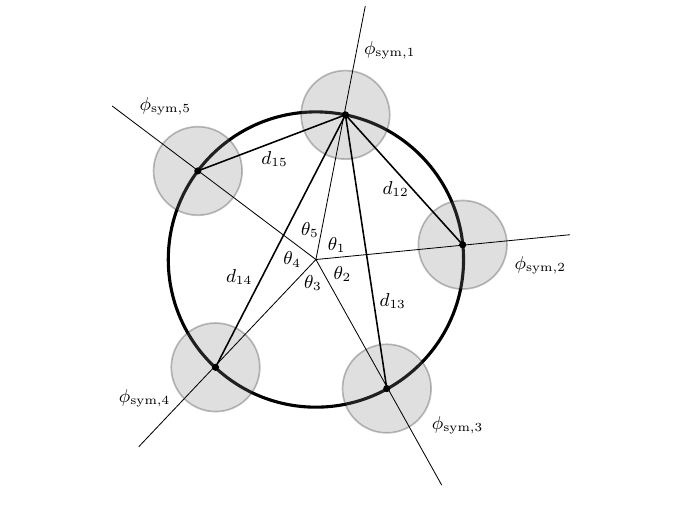}
\caption{To the proof of Theorem~\ref{thm:circleopt}}
\label{fig:sectors}
\end{figure}
 % -------------- %
As before we construct the trial function $\phi_Y$ as a `array of beads'; we start from the restriction of $\phi_\mathrm{sym}$ to the ball $B_\rho(y_1)$ calling it $\phi_{\mathrm{sym},1}$ and use it to create $\phi_{\mathrm{sym},j},\, j=2,\dots,N$, by rotating this function on the angle $\sum_{i=1}^{j-1} \theta_i$ around the origin. For $Y=Y_\mathrm{sym}$ the left-hand side of \eqref{BSring} vanishes by construction, hence it is sufficient to prove that
 % ------------- %
 $$ %\begin{equation} \label{BSring2}
(\phi_Y, K_{V,Y}(-\kappa^2)\phi_Y) - (\phi_\mathrm{sym}, K_{V,Y_\mathrm{sym}}(-\kappa^2)\phi_\mathrm{sym}) >0
 $$ %\end{equation}
 % ------------- %
holds for any $\kappa>0$, in particular, for $\kappa=\kappa_\mathrm{sym}$, or explicitly
 % ------------- %
 \begin{align*}
& \frac{1}{2\pi} \sum_{i,j=1}^N \Big\{ \int_{B_\rho(0)} \int_{B_\rho(0)} \phi_\mathrm{sym}(\xi) V^{1/2}(\xi)\, K_0(\kappa|y_i+\xi-y_j-\xi'|)\, \\[.15em] & \qquad\qquad \times V^{1/2}(\xi') \phi_\mathrm{sym}(\xi')\,\D\xi\,\D\xi' \\[.3em] & -  \int_{B_\rho(0)} \int_{B_\rho(0)} \phi_\mathrm{sym}(\xi) V^{1/2}(\xi)\, K_0(\kappa|y_i^{(0)}\!+\xi-y_j^{(0)}\!-\xi'|) \\[.15em] & \qquad\qquad \times  V^{1/2}(\xi') \phi_\mathrm{sym}(\xi')\,\D\xi\,\D\xi' \Big\} > 0
 \end{align*}
 % ------------- %
We denote $d_{ij} := |y_i-y_j|$ and $d_{ij}^{(0)} := |y_i^{(0)}-y_j^{(0)}|$ and write the first part of the above expression as $\sum_{i,j=1}^N\tilde{G}_{i\kappa}(d_{ij})$, in the second one $d_{ij}$ is replaced by $d_{ij}^{(0)}$; we can do that because the expressions obtained by the integration over the balls depend only on the distances between the their centers. The sought inequality then takes the form
 % ------------- %
$$
\sum_{i,j=1}^N \tilde{G}_{i\kappa}(d_{ij}) > \sum_{i,j=1}^N \tilde{G}_{i\kappa}(d_{ij}^{(0)}),
$$
 % ------------- %
and rearranging the summation order, we have to check that
 % ------------- %
$$
F(d_{ij}) := \sum_{m=1}^{[N/2]} \sum_{|i-j|=m} \big[ \tilde{G}_{i\kappa}(d_{ij})- \tilde{G}_{i\kappa}\big(d_{ij}^{(0)}\big)\big] > 0
$$
 % ------------- %
holds for every family $\{d_{ij}\}$ which is not congruent with $\{d_{ij}^{(0)}\}$.

The resolvent kernel contained in the expression is a convex function of its argument, and since $|\xi-\xi'|<2\rho<d_{ij}$, the function $d_{ij}\mapsto|y_i+\xi-y_j-\xi'|$ is increasing and concave. Consequently, $d_{ij}\mapsto K_0(\kappa|y_i+\xi-y_j-\xi'|)$ is convex again for any $\xi,\xi'\in B_\rho(0)$, and being integrated with the positive weight the result will be again convex. This makes it possible to apply Jensen's inequality which yields
 % ------------- %
 \begin{equation} \label{BSchords}
F(d_{ij}) \ge \sum_{m=1}^{[N/2]} \nu_n \Big[ \tilde{G}_{i\kappa}\Big( \frac{1}{\nu_n} \sum_{|i-j|=m} d_{ij} \Big) - \tilde{G}_{i\kappa}\big(d_{i,i+m}^{(0)}\big) \Big],
 \end{equation}
 % ------------- %
where $\nu_n$ is the number of distinct line segments connecting the points $y_i$ and $y_{i+m}$ for $m=1,\dots,N$, that is, $\nu_n=N$ except the case when $N$ is even and $m=\frac12 N$ where $\nu_n= \frac12 N$.

To prove that the right-hand side of \eqref{BSchords} is positive we use the fact the convexity is not the only property which $\tilde{G}_{i\kappa}(\cdot)$ inherited from the resolvent kernel; since $d_{ij}\mapsto|y_i+\xi-y_j-\xi'|$ is increasing, the integrated function is (strictly) decreasing which means that it is only necessary to check that
 % ------------- %
 \begin{equation} \label{BSposit}
\frac{1}{\nu_n} \sum_{|i-j|=m} d_{ij} < d_{i,i+m}^{(0)}
 \end{equation}
 % ------------- %
for any fixed $i$. Denoting $\beta_{ij}=\sum_{k=i}^{j-1} \theta_k$, we have $d_{ij}= 2\sin\frac12\beta_{ij}$ and $d_{i,i+m}^{(0)} = 2\sin\frac{\pi m}{N}$, and since the sine is concave in $(0,\pi)$, we can use Jensen's inequality for concave function which gives
 % ------------- %
$$
\frac{1}{\nu_n} \sum_{|i-j|=m} 2\sin\frac12\beta_{ij} < 2\sin\Big( \frac{1}{\nu_n} \sum_{|i-j|=m} \frac12\beta_{ij} \Big) = 2\sin\frac{\pi m}{N} = d_{i,i+m}^{(0)}
$$
 % ------------- %
for those families $\{d_{ij}\}$ of circle chords which are not congruent with $\{d_{ij}^{(0)}\}$; this concludes the proof.
\end{proof}

 % ------------- %
 \begin{remark} \label{rem:3D circle}
For simplicity, we have formulated the claim and its proof in the two-dimensional setting but the argument extends easily to arrays $Y\subset\R^3$ situated on a planar circle. Note also that the symmetry of the potential $V$ can abandoned as long as all the potential wells involved can be obtained one from another by rotations.
 \end{remark}
 % ------------- %

Beyond this simple extension there are more complicated questions. To begin with, the maximizing configuration in Theorem~\ref{thm:circleopt} places the disk centers at vertices of regular polygon of the perimeter $2NR\sin\frac{\pi}{N}$. It is then natural to ask about the maximization within a wider class of sets $Y$ in the setting analogous to that used in Sec.~\ref{s:bentchain}.
 % ------------- %
\begin{conjecture} \label{conj:loopopt}
Suppose that the points of $Y$ are on a loop $\Gamma$ of a fixed length in $\R^\nu,\, \nu=2,3$, equidistantly in the arc length variable, and the balls $B_\rho(y_i)$ do not overlap. Then $\epsilon_1(H_{V,Y})=\inf\sigma(H_{V,Y})$ is maximized, uniquely up to Euclidean transformations, by a planar regular polygon of $\#Y$ vertices.
\end{conjecture}
 % ------------- %

The next question is much harder. We again fix the manifold on which points of $Y$ are allowed to be, but this time not as a curve, but as a sphere in $\R^3$, and ask about the configurations optimizing the ground state energy of $H_{V,Y}$. This problem has the flavor of the celebrated Thomson problem \cite{Th04, Twiki}, not fully solved after more than a century of efforts, except that we seek a maximizing, not minimizing configuration. What one can realistically hope for is the solution in particular cases of low $N=\#Y$:

 % ------------- %
\begin{conjecture} \label{conj:thomson}
Let the point of $Y$ be arranged on the a sphere in such a way that the balls $B_\rho(y_i)$ do not overlap. Then $\epsilon_1(H_{V,Y})=\inf\sigma(H_{V,Y})$ is maximized, uniquely up to Euclidean transformations, by the following five `equilateral' configurations:
 % ------------- %
\begin{itemize}
\setlength{\itemsep}{0pt}
\item three \emph{simplices}, with $N=2$ (a pair antipodal points), $N=3$ (equilateral triangle), and $N=4$ (tetrahedron),
 % ------------- %
\item \emph{octahedron} with $N=6$,
 % ------------- %
\item \emph{icosahedron} with $N=12$.
\end{itemize}
 % ------------- %
\end{conjecture}
 % ------------- %
Both the conjectures we have stated are motivated by the fact that in the singular limit discussed in Section~\ref{s:shrinking} the corresponding optimisation results are known to be valid as demonstrated in \cite{Ex19}.

%%%%%%%%%%%%%%%%%%%%%
\section{Conclusions}
\label{s:conc}

The question about spectral properties of Schr\"odinger operators with potentials mixing a local order with a nontrivial geometry is rich and the current discussion just scratched the surface while suggesting various open problems. One may ask, for instance, about finer properties of the curvature-induced spectrum in relation to the geometry of the array and the shape of the single cell potential, such as the spectral counting function, the weak-deformation asymptotic behavior of the ground state, etc. At the same time, there are numerous generalizations one can think of. In addition to the asymmetry of $V$ mentioned in Remarks~\ref{rem:asymmetry} and \ref{rem:3D circle}, they include sign-changing potentials, replacing the chain $Y$ by more complicated lattices, or a quasi-periodic arrangement of the building blocks.

Another question of interest concerns the influence of a magnetic field. The two-dimensional Landau Hamiltonian perturbed by a straight periodic array of point interactions is known to have the spectrum containing the unperturbed Landau level and absolutely continuous bands between them \cite{EJK99}. The point part is not likely to persist if the singular interactions will be replaced by regular potentials but the absolute continuity is expected to be preserved. One can again ask whether some geometric perturbations will give rise to a discrete spectrum. From the point of view of applications, it is also important to find out whether a part of the absolutely continuous spectrum can survive random perturbations.

Conjectures~\ref{conj:loopopt} and \ref{conj:thomson} are not the only optimisation problems one can address in systems with finite numbers of potential wells. As long as we suppose that the balls supporting the individual potentials do not overlap, it is also natural to ask about the configurations that \emph{minimize} the ground state. Under our assumptions about the potential $V$ the answer can be easily found for a few smallest values of $N$, for larger ones -- or with additional geometric constraints -- the task may be considerably more difficult.

%We provide this list, no doubt incomplete, in belief that open problems are the best birthday gift to a mathematician. It
This list is no doubt incomplete and could go on, but we prefer to stop here and leave the continuation open.

%%%%%%%%%%%%%%%%%%%%%%%%%%%%%%
\subsection*{Acknowledgements}

The work was supported by the Czech Science Foundation within the project 21-07129S. The author is grateful to Vladimir Lotoreichik for a useful discussion and pointing out a flaw in the first version of the proof of Theorem~\ref{thm:circleopt}, and to the referees for the comments which helped to improve the presentation.

\end{document}